\documentclass[11pt]{amsart}
\usepackage{fullpage}
\usepackage{amssymb, amsmath, wasysym, mathrsfs, mathtools, color, stmaryrd}
\usepackage[all, cmtip]{xy}
\usepackage{url}

\usepackage{csquotes}
\usepackage{tikz-cd}

\definecolor{hot}{RGB}{65,105,225}
\usepackage[pagebackref=true,colorlinks=true, linkcolor=hot ,  citecolor=hot, urlcolor=hot]{hyperref}

\theoremstyle{definition}
\newtheorem{defi}{Definition}[section]

\theoremstyle{plain}
\newtheorem{thm}{Theorem}[section]
\newtheorem{prop}{Proposition}[section]
\newtheorem{lem}{Lemma}[section]
\newtheorem{coro}{Corollary}[section]

\theoremstyle{remark}
\theoremstyle{remark}
\newtheorem{rem}{Remark}[section]
\newtheorem*{ackn}{Acknowledgements}

\newtheorem*{thm*}{Theorem}

\newcommand{\Sp}{\operatorname{Spec}}

\newcommand{\dg}{\operatorname{\mathbf{dgArt}\!}{}}
\newcommand{\Ens}{\operatorname{SEns\!}{}}
\newcommand{\ens}{\operatorname{\mathbf{SEns}\!}{}}
\newcommand{\Def}{\operatorname{Def\!}{}}
\newcommand{\Hom}{\operatorname{Hom\!}{}}

\newcommand{\Mo}{\operatorname{Mod\!}{}}
\newcommand{\Mod}{\operatorname{\mathbf{Mod}\!}{}}

\newcommand{\dSc}{\operatorname{\mathbf{dSch}\!}{}}

\newcommand{\cdga}{\operatorname{cdga\!}{}}
\newcommand{\cdg}{\operatorname{\mathbf{cdga}\!}{}}
\newcommand{\rt}{\operatorname{\mathbb{R}\Gamma\!}{}}

\newcommand{\Spec}{\operatorname{Spec\!}{}}
\newcommand{\free}{\operatorname{Free\!}{}}
\newcommand{\calg}{\operatorname{CAlg\!}{}}

\newcommand{\Calg}{\operatorname{\mathbf{CAlg}\!}{}}

\newcommand{\Li}{\operatorname{Lie\!}{}}
\newcommand{\Lie}{\operatorname{\mathbf{Lie}\!}{}}
\newcommand{\Cn}{\operatorname{Cn\!}{}}

\newcommand{\Map}{\operatorname{Map\!}{}}

\newcommand{\se}{\operatorname{\mathbf{Sets}\!}{}}

\newcommand{\coli}{\operatorname{colim\!}{}}
\newcommand{\Art}{\operatorname{\mathbf{Art}\!}{}}

\title{Semi-prorepresentability of formal moduli problems and equivariant structures}

\author{An-Khuong DOAN}
\address{An-Khuong DOAN, IMJ-PRG, UMR 7586, Sorbonne Université,  Case 247, 4 place Jussieu, 75252 Paris Cedex 05, France}
\email{an-khuong.doan@imj-prg.fr }
\thanks{ }

\begin{document}


\date{June 26, 2023.}

\dedicatory{ }


\begin{abstract}  We generalize the notion of semi-universality in the classical deformation problems to the context of derived deformation theories. A criterion for a formal moduli problem to be semi-prorepresentable is produced. This can be seen as an analogue of Schlessinger's conditions for a functor of Artinian rings to have a semi-universal element. We also give a sufficient condition for a semi-prorepresentable formal moduli problem to admit a $G$-equivariant structure in a sense specified below, where $G$ is a linearly reductive group. Finally, by making use of these criteria, we derive many classical results including the existence of ($G$-equivariant) formal semi-universal deformations of algebraic schemes and that of complex compact manifolds.
\end{abstract}

\maketitle
\tableofcontents

\section{Introduction} \label{sec3.1} The theory of deformations of algebraic schemes with algebraic group actions is first studied by the pioneering work of Pinkham (see \cite{6}) in which affine cones with $\mathbb{G}_m$-actions are taken into account. Six years later, Rim obtains a far-reaching result which claims that if $G$ is a linearly reductive group acting algebraically on an algebraic scheme $X_0$ where $X_0$ is supposed to be either an affine scheme with at most isolated singularities or a complete algebraic variety then a $G$-equivariant formal semi-universal deformation of $X_0$ exists, unique up to $G$-equivariant isomorphism (see \cite{Rim}). In the language of functors
of Artin rings, this result can be rephrased as follows. Let $k$ be an algebraically closed field and $\Art_k$ (resp. $\widehat{\Art}_k$) be the category of local artinian $k$-algebras (resp. complete local noetherian $k$-algebras) with residue field $k$. The functor $F_{X_0}$: $\Art_k \rightarrow \se$ which associates to each local artinian $k$-algebra $A$, the set of flat morphisms of schemes $X\rightarrow \Spec(A)$ with an isomorphism $$X\times_{\Spec(A)}\Spec(k) \cong X_0$$ has a formal semi-universal element, i.e. there exists a pro-object $R$ in $\widehat{\Art}_k$ and an element $\hat{u} \in \widehat{F_{X_0}}(R)$ such that the morphism of functors $$\Hom_{\widehat{\Art}_k}(R,-)\rightarrow F_{X_0}$$ defined by $\hat{u}$ is smooth and such that  $$\Hom_{\widehat{\Art}_k}(R,k[\epsilon]/(\epsilon^2))\rightarrow F_{X_0}(k[\epsilon]/(\epsilon^2))$$ is bijective, where $\widehat{F_{X_0}}$ is the extension of $F_{X_0}$ on $\widehat{\Art}_k$ (see \cite[$\S$2.2]{12} for more details) and $k[\epsilon]/(\epsilon^2)$ is the ring of dual numbers. Furthermore, this formal semi-universal element can be made $G$-equivariant. A recently-constructed counter-example in \cite{1} has shown that the reductiveness assumption on $G$ turns out to be optimal.  In general, hardly is $F_{X_0}$ prorepresentable by a pro-object due to the existence of non-trivial automorphisms of $X_0$ as always. Therefore, the smooth morphism $$\Hom_{\widehat{\Art}_k}(R,-)\rightarrow F_{X_0}$$ can be considered the best formal approximation of $F_{X_0}$ that we can expect. A similar result on the existence of $G$-equivariant Kuranishi family of compact complex manifolds is obtained as well in \cite{2}. In other words, the functor of Artin rings associated to the deformation problem of a given complex compact manifold admits a $G$-equivariant semi-universal element. The main difference here is that on the analytic side, all deformations are required to be convergent while on the algebraic side, they might exist formally.

Besides, a well-known philosophy of Drindfeld states that: ``If $X$ is a moduli space over a field $k$ of characteristic zero, then a formal neighborhood of any point $x \in X$ is controlled by a differential graded Lie algebra" of which Lurie's paper  \cite{4} and Pridham's one \cite{8} have given a rigorous formulation. Namely, instead of working with $\Art_k$, they work with the category of differential graded commutative artinian augmented $k$-algebras, denoted by $\dg_k$ and a formal moduli problem in his sense is defined to be a functor from $\dg_k \rightarrow \ens$ satisfying certain exactness conditions, where $\ens$ is the $\infty$-category of simplicial sets. Then they prove that there is an equivalence of $\infty$-categories between the homotopic category of formal moduli problems and that of differential graded Lie algebras. Furthermore, the prorepresentability (which corresponds to the notion of universality in the classical sense) of a formal moduli problem is reduced to checking some cohomological conditions on its associated differential graded Lie algebra, which is feasible for most of natural formal moduli problems that we encounter in reality. This can be viewed as an extremely astonishing generalization of Schlessinger's work on functors of artinian rings (cf. \cite{11}). 

However, the notion of semi-universality apparently does not exist in the $\infty$-framework. A germ of this notion might be retrieved in Manetti's work (cf. \cite[Def. 2.1]{35}) but it has not still been fully simplicially generalized on the target.  Therefore, in this paper, our aim is to fully define such a notion which we shall call ``semi-prorepresentability" (cf. Definition \ref{d2.2}). This notion should generalize the notion of semi-universality given by M. Schlessinger. Then we prove the semi-prorepresentability for a class of formal moduli problems of which the formal moduli problem $\Def_{X_0}$ associated to derived deformations of algebraic schemes or to those of complex compact manifolds (which is a natural extension of the functor $F_{X_0}$ in the derived literature) is a typical example (cf. Theorem \ref{t2.2}). This gives us an algebraic way to recover the formal existence of semi-universal deformations in the classical setting. At last, we will prove a theorem of Rim's type. More precisely, we would like to provide a $G$-equivariant structure to the pro-object in $\dg_k$, which semi-prorepresents $\Def_{X_0}$. Inspired by the spirit of Lurie's equivalence, we shall carry things out on the corresponding differential graded Lie algebra (see Definition \ref{d3.3.3} and Theorem \ref{t2.3}). Once again, Rim's result  is just an immediate corollary of this. All of this could be  also somehow regarded as an equivariant version of the classicial Nijenhuis-Richardson's trick (cf. \cite{26}). Moreover, while revising the manuscript, we were informed of the existence of the beautiful paper \cite{25} where similar results might be obtained, using a much more modern approach: $L_\infty$-models.

Let us now outline the organization of this paper. We first, in $\S$\ref{sec3.2}, give an overview of the $\infty$-equivalence between the category of formal moduli problems and that of differential graded Lie algebras. This can be served as a quick gentle introduction to Lurie's work in \cite{4}.  In $\S$\ref{sec3.3}, we shall introduce the notion of semi-prorepresentability and prove that any formal moduli problem in Lurie's sense is semi-prorepresentable. If further the associated differential graded Lie algebra of this formal moduli problem is equipped with an appropriate action of some linearly reductive group $G$, we show that the corresponding semi-prorepresentable pro-object can be equipped with a versal compatible $G$-action. In $\S$\ref{sec3.4}, we recall the famous differential graded Lie algebras which control analytic deformations of a given complex compact manifold or algebraic deformations of a given scheme. Finally, the existence of ($G$-equivariant) formal semi-universal deformation of algebraic schemes and that of complex compact manifolds are just immediate consequences of what we have done in $\S$\ref{sec3.3}.
\newline

\textbf{Conventions and notations}: \begin{enumerate}
\item[$\bullet$] A field of characteristic $0$ will be always denoted by $k$.
\item[$\bullet$] dgLa is the abbreviation of differential graded Lie $k$-algebra while cgda means commutative differential graded augmented $k$-algebra.
\item[$\bullet$] $\Mo_k$ is the category of cochain complexes of $k$-modules and $\Mod_k$ is the corresponding $\infty$-category.
\item[$\bullet$] $\Li_k$ is the category of differential graded Lie $k$-algebras and $\Lie_k$ is the corresponding $\infty$-category.
\item[$\bullet$] $\cdga_k$ is the category of commutative differential graded augmented $k$-algebras and $\cdg_k$ is the corresponding $\infty$-category.
\item[$\bullet$] $\dg_k$ denotes the full sub-category of $\cdga_k$ consisting of commutative differential graded artinian algebras cohomologically concentrated in non-positive degrees.
\item[$\bullet$] $\Art_k$ denotes the category of local artinian $k$-algebras with residue field $k$.
\item[$\bullet$] $\ens$ is the category of simplicial sets.
\item[$\bullet$] fmp is the abbreviation of formal moduli problem.
\item[$\bullet$] $\mathcal{FMP}$  is the homotopic category of formal moduli problems. 
\end{enumerate}

\begin{ackn}  We would like to profoundly thank Prof. Bertrand To\"{e}n for explaining some notions in \cite{13} and \cite{14} and Prof. Julien Grivaux for many precious discussions. We are specially thankful to the referee whose dedicated
work led to a great amelioration of this paper.
\end{ackn}

\section{Formal moduli problem revisited} \label{sec3.2}
\subsection{Presentable $\infty$-categories} A glimpse on presentable $\infty$-categories is provided in this section. Let $\mathbf{\Delta}$ be the category of finite ordinal numbers with order-preserving maps between them. Concretely, the objects of $\mathbf{\Delta}$ are strings 
$$\mathbf{n}:\; 0\rightarrow 1\rightarrow \cdots \rightarrow n$$ where $n$ is a positive integer and morphisms of $\mathbf{\Delta}$ are order-preserving set functors $\mathbf{m}\rightarrow\mathbf{n}$. For each $\mathbf{n}\in \mathbf{\Delta}$, consider the following morphisms:
\begin{align*}
d^i : \mathbf{n} -\mathbf{1}&\rightarrow \mathbf{n} \\
 \left (0\rightarrow 1\rightarrow \cdots \rightarrow n-1  \right )&\mapsto \left (0\rightarrow 1\rightarrow \cdots \rightarrow i-1\rightarrow i+1\rightarrow\cdots\rightarrow n  \right )
\end{align*} and \begin{align*}
s^j : \mathbf{n} +\mathbf{1}&\rightarrow \mathbf{n} \\
 \left (0\rightarrow 1\rightarrow \cdots \rightarrow n+1  \right )&\mapsto \left (0\rightarrow 1\rightarrow \cdots \rightarrow j  {\rightarrow} j\rightarrow\cdots\rightarrow n  \right ).
\end{align*} The former ones are called cofaces while the latter ones are called codegeneracies. They satisfies the following cosimplicial identities
\begin{equation}\label{e0.01 }
\begin{cases}
d^jd^i=d^id^{j-1} & \text{ if } i<j \\ 
 s^jd^i=d^is^{j-1}& \text{ if } i<j \\ 
 s^jd^j=\mathrm{Id}=s^jd^{j+1}&  \\ 
 s^jd^i=d^{i-1}s^j& \text{ if } i>j+1\\ 
s^js^i=s^is^{j+1} & \text{ if } i\leq j .
\end{cases}
\end{equation}
The maps $d^i$, $s^j$ together with these relations constitute a set of generators and relations for $\mathbf{\Delta}$ (cf. \cite{34}).\begin{defi} A simplicial set is a contravariant functor $X:\; \mathbf{\Delta}\rightarrow \se$. A map of simplicial sets $f:\; X \rightarrow Y$ is simply a natural transformation of contravariant set-valued functors defined over $\mathbf{\Delta}$. 
\end{defi} Using the generators $d^i$, $s^j$ and the relations (\ref{e0.01 }), to give a simplicial set $Y$ is equivalent to giving sets $Y_n$, $n\geq 0$ together with maps
$$\begin{cases}
d_i:Y_n\rightarrow Y_{n-1}, & 0\leq i \leq n \;(\text{faces}) \\ 
 s_j:Y_n\rightarrow Y_{n+1},& 0\leq j \leq n \;(\text{degeneracies})
\end{cases}$$ satisfying the simplicial identities
$$\begin{cases}
d_id_j=d_{j-1}d_i & \text{ if } i<j \\ 
 d_is_j=s_{j-1}d_i& \text{ if } i<j \\ 
 d_js_j=\mathrm{Id}=d_{j+1}s_j&  \\ 
 d_is_j=s_jd_{i-1}& \text{ if } i>j+1\\ 
s_is_j=s_{j+1}s_i & \text{ if } i\leq j .
\end{cases}$$ We denote the category of simplicial sets by $\Ens$ and refer the reader to \cite{33} for a complete study of this category.
\begin{defi}\begin{enumerate}
\item[(1)] The standard $n$-simplex in the category $\Ens$ is defined  by
$$\Delta^n=\Hom_{\mathbf{\Delta}}(\cdot, \mathbf{n}).$$
\item[(2)] Denote by $\iota_n$ the standard simplex $\mathrm{Id}_\mathbf{n}\in \Hom_{\mathbf{\Delta}}(\mathbf{n},\mathbf{n})$. For $0\leq k \leq n$, the $k$-horn $\Lambda_{k}^n$ of $\Delta^n$ is the union of all the faces $d_j(\iota_n)$ except $d_k(\iota_n)$.
\end{enumerate}
\end{defi}
\tikzset{every picture/.style={line width=0.75pt}} 
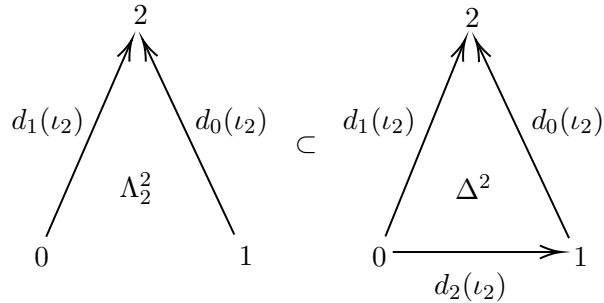
\begin{center}

\begin{tikzpicture}[x=0.75pt,y=0.75pt,yscale=-1,xscale=1]

\draw    (328.13,188.38) -- (367.69,90.69) ;
\draw [shift={(368.44,88.84)}, rotate = 472.04] [color={rgb, 255:red, 0; green, 0; blue, 0 }  ][line width=0.75]    (10.93,-3.29) .. controls (6.95,-1.4) and (3.31,-0.3) .. (0,0) .. controls (3.31,0.3) and (6.95,1.4) .. (10.93,3.29)   ;
\draw    (420.75,189.28) -- (374.49,90.92) ;
\draw [shift={(373.64,89.11)}, rotate = 424.81] [color={rgb, 255:red, 0; green, 0; blue, 0 }  ][line width=0.75]    (10.93,-3.29) .. controls (6.95,-1.4) and (3.31,-0.3) .. (0,0) .. controls (3.31,0.3) and (6.95,1.4) .. (10.93,3.29)   ;
\draw    (157.07,188.55) -- (199.59,90.67) ;
\draw [shift={(200.39,88.84)}, rotate = 473.48] [color={rgb, 255:red, 0; green, 0; blue, 0 }  ][line width=0.75]    (10.93,-3.29) .. controls (6.95,-1.4) and (3.31,-0.3) .. (0,0) .. controls (3.31,0.3) and (6.95,1.4) .. (10.93,3.29)   ;
\draw    (251.95,187.46) -- (206.46,90.65) ;
\draw [shift={(205.61,88.84)}, rotate = 424.83000000000004] [color={rgb, 255:red, 0; green, 0; blue, 0 }  ][line width=0.75]    (10.93,-3.29) .. controls (6.95,-1.4) and (3.31,-0.3) .. (0,0) .. controls (3.31,0.3) and (6.95,1.4) .. (10.93,3.29)   ;
\draw    (331.84,196.26) -- (414.98,196.26) ;
\draw [shift={(416.98,196.26)}, rotate = 180] [color={rgb, 255:red, 0; green, 0; blue, 0 }  ][line width=0.75]    (10.93,-3.29) .. controls (6.95,-1.4) and (3.31,-0.3) .. (0,0) .. controls (3.31,0.3) and (6.95,1.4) .. (10.93,3.29)   ;

\draw (319.9,192.4) node [anchor=north west][inner sep=0.75pt]  [rotate=-359.96]  {$0$};
\draw (421.47,192.4) node [anchor=north west][inner sep=0.75pt]  [rotate=-359.96]  {$1$};
\draw (366.88,71.95) node [anchor=north west][inner sep=0.75pt]  [rotate=-359.96]  {$2$};
\draw (149.61,192.4) node [anchor=north west][inner sep=0.75pt]  [rotate=-359.96]  {$0$};
\draw (252.67,191.76) node [anchor=north west][inner sep=0.75pt]  [rotate=-359.96]  {$1$};
\draw (199.58,70.11) node [anchor=north west][inner sep=0.75pt]  [rotate=-359.96]  {$2$};
\draw (281,136.38) node [anchor=north west][inner sep=0.75pt]    {$\subset $};
\draw (192.59,154.71) node [anchor=north west][inner sep=0.75pt]    {$\Lambda _{2}^{2}$};
\draw (361.79,155.75) node [anchor=north west][inner sep=0.75pt]    {$\Delta ^{2}$};
\draw (400.24,121.69) node [anchor=north west][inner sep=0.75pt]    {$d_{0}( \iota _{2})$};
\draw (305.09,121.65) node [anchor=north west][inner sep=0.75pt]    {$d_{1}( \iota _{2})$};
\draw (351.17,205.18) node [anchor=north west][inner sep=0.75pt]    {$d_{2}( \iota _{2})$};
\draw (138.06,121.65) node [anchor=north west][inner sep=0.75pt]    {$d_{1}( \iota _{2})$};
\draw (230.68,121.65) node [anchor=north west][inner sep=0.75pt]    {$d_{0}( \iota _{2})$};

\end{tikzpicture}

\end{center}

\begin{defi} An $\infty$-category is a simplicial set $K$ which has the following property: for any $0< k< n$, any map $f_0: \Lambda_{k}^n \rightarrow K$ admits an extension $f:\Delta^n \rightarrow K$ 
\tikzset{every picture/.style={line width=0.75pt}} 
\begin{center}

\begin{tikzpicture}[x=0.75pt,y=0.75pt,yscale=-1,xscale=1]

\draw    (194.35,75.04) -- (305.21,75.04) ;
\draw [shift={(307.21,75.04)}, rotate = 180] [color={rgb, 255:red, 0; green, 0; blue, 0 }  ][line width=0.8]    (10.93,-3.29) .. controls (6.95,-1.4) and (3.31,-0.3) .. (0,0) .. controls (3.31,0.3) and (6.95,1.4) .. (10.93,3.29)   ;
\draw    (185.79,82.56) -- (185.79,141.08) ;
\draw [shift={(185.79,143.08)}, rotate = 270] [color={rgb, 255:red, 0; green, 0; blue, 0 }  ][line width=0.75]    (10.93,-3.29) .. controls (6.95,-1.4) and (3.31,-0.3) .. (0,0) .. controls (3.31,0.3) and (6.95,1.4) .. (10.93,3.29)   ;
\draw  [dash pattern={on 4.5pt off 4.5pt}]  (193.57,150.22) -- (305.5,83.21) ;
\draw [shift={(307.21,82.18)}, rotate = 509.09] [color={rgb, 255:red, 0; green, 0; blue, 0 }  ][line width=0.75]    (10.93,-4.9) .. controls (6.95,-2.3) and (3.31,-0.67) .. (0,0) .. controls (3.31,0.67) and (6.95,2.3) .. (10.93,4.9)   ;
\draw    (220.04,90.45) .. controls (199.02,96.47) and (206.18,111.18) .. (225.56,111.34) .. controls (244.46,111.5) and (253.81,96.49) .. (232.62,90.87) ;
\draw [shift={(230.93,90.45)}, rotate = 372.7] [color={rgb, 255:red, 0; green, 0; blue, 0 }  ][line width=0.75]    (10.93,-3.29) .. controls (6.95,-1.4) and (3.31,-0.3) .. (0,0) .. controls (3.31,0.3) and (6.95,1.4) .. (10.93,3.29)   ;

\draw (170.09,61.06) node [anchor=north west][inner sep=0.75pt]    {$\Lambda _{k}^{n}$};
\draw (310.41,65.18) node [anchor=north west][inner sep=0.75pt]    {$K$};
\draw (173.2,149.26) node [anchor=north west][inner sep=0.75pt]    {$\Delta ^{n}$};
\draw (171.2,101.27) node [anchor=north west][inner sep=0.75pt]    {$\iota $};
\draw (236.36,52.42) node [anchor=north west][inner sep=0.75pt]    {$f_{0}$};
\draw (238.14,129.82) node [anchor=north west][inner sep=0.75pt]    {$f$};

\end{tikzpicture}

\end{center}
(cf. \cite[Def. 1.1.2.4]{30}). A functor (often called $\infty$-functor) between two $\infty$-categories is simply a map of simplicial sets.  
 \end{defi}
To end this section, we introduce the notion of presentability of $\infty$-categories. (cf.  \cite[Def. 5.4.2.1, Prop. 5.4.2.2 and Def. 5.5.0.1]{30}).
\begin{defi} \label{d3.2.4} Let $\mathcal{C}$  be a category (or an $\infty$-category). We say that $\mathcal{C}$ is presentable if $\mathcal{C}$ admits small colimits and
is generated under small colimits by a set of $\kappa$-compact objects, for some regular cardinal number $\kappa$. Here, an object $C \in \mathcal{C}$ is said to be $\kappa$-compact if the functor $\Hom_{\mathcal{C}}(C, -)$ preserves $\kappa$-filtered colimits 
\end{defi} 
\begin{rem} We often omit the cardinal number $\kappa$ and say simply ``compact" and ``filtered"  for simplicity.
\end{rem} 

There is a general effective method to construct presentable $\infty$-categories via combinatorial model categories (see \cite{32} for the notion of combinatorial model category) and Dwyer-Kan simplicial localization (\cite{31}), which we shall use several times in the sequel. We recall it here for completeness. Let $\mathcal{C}$ be a model category and $\mathrm{N}(\mathcal{C})$ its associated nerve category (cf. \cite[Def. 1.1.5.5]{30}). Concretely, the simplices of $\mathrm{N}(\mathcal{C})$ can be explicitly described as follows.
\begin{enumerate}
\item[$\bullet$] $0$-simplices are objects of $\mathcal{C}$,
\item[$\bullet$] $1$-simplices are morphisms of $\mathcal{C}$.
\item[] $\cdots$
\item[$\bullet$] $n$-simplices are strings of $n$ composable morphisms $$C_0\overset{f_1}{\rightarrow}C_1\overset{f_2}{\rightarrow}\cdots\overset{f_{n-1}}{\rightarrow} C_{n-1}\overset{f_n}{\rightarrow}C_n$$
which the face map $d_i$ and the degeneracy map $s_j$ carry to 
$$C_0\overset{f_1}{\rightarrow}C_1\overset{f_2}{\rightarrow}\cdots\overset{f_{i-1}}{\rightarrow} C_{i-1}\overset{f_{i+1}\circ f_i}{\rightarrow}C_{i+1}\overset{f_{i+2}}{\rightarrow}  \cdots\overset{f_{n-1}}{\rightarrow} C_{n-1}\overset{f_n}{\rightarrow}C_n$$ and $$C_0\overset{f_1}{\rightarrow}C_1\overset{f_2}{\rightarrow}\cdots\overset{f_{j}}{\rightarrow} C_{j}\overset{\mathrm{Id}_{C_j}}{\rightarrow}C_{j}\overset{f_{j+1}}{\rightarrow}  \cdots\overset{f_{n-1}}{\rightarrow} C_{n-1}\overset{f_n}{\rightarrow}C_n,$$ respectively.
 
\end{enumerate} By formally inverting the class $W_{\mathcal{C}}$ of weak equivalences in $\mathcal{C}$, we obtain a category $\mathrm{N}(\mathcal{C})[W_{\mathcal{C}}^{-1}]$ which is the associated $\infty$-category of $\mathcal{C}$. The presentability of $\mathrm{N}(\mathcal{C})[W_{\mathcal{C}}^{-1}]$ follows immediately from the following theorem (cf. \cite[Prop. 1.3.4.22]{29}).
\begin{thm} Let $\mathcal{C}$ be a combinatorial model category. Then the associated $\infty$-category of $\mathcal{C}$ is presentable.
\end{thm}
As fundamental examples, we shall mention the associated presentable $\infty$-categories of the category $\Ens$ of simplicial sets, of the category of differential graded Lie algebras and of the category of commutative differential graded augmented $k$-algebras.

 \begin{prop} The category $\Ens$ of simplicial sets admits a combinatorial model category structure where 
\begin{enumerate}
\item[($W$)] A map of simplicial sets $f: X \rightarrow Y$ is a weak equivalence if and only if its geometric realization is a weak homotopy equivalence of topological spaces.
\item[($F$)] A map of simplicial sets $f: X \rightarrow Y$ is a fibration if and only if it satisfies the Kan condition, i.e. for any $0\leq k \leq n$ and any diagram 
\begin{center}
\begin{tikzpicture}[x=0.75pt,y=0.75pt,yscale=-1,xscale=1]

\draw    (195.35,75.04) -- (306.21,75.04) ;
\draw [shift={(308.21,75.04)}, rotate = 180] [color={rgb, 255:red, 0; green, 0; blue, 0 }  ][line width=0.75]    (10.93,-3.29) .. controls (6.95,-1.4) and (3.31,-0.3) .. (0,0) .. controls (3.31,0.3) and (6.95,1.4) .. (10.93,3.29)   ;
\draw    (185.79,82.56) -- (185.79,141.08) ;
\draw [shift={(185.79,143.08)}, rotate = 270] [color={rgb, 255:red, 0; green, 0; blue, 0 }  ][line width=0.75]    (10.93,-3.29) .. controls (6.95,-1.4) and (3.31,-0.3) .. (0,0) .. controls (3.31,0.3) and (6.95,1.4) .. (10.93,3.29)   ;
\draw  [dash pattern={on 4.5pt off 4.5pt}]  (193.57,150.22) -- (305.5,83.21) ;
\draw [shift={(307.21,82.18)}, rotate = 509.09] [color={rgb, 255:red, 0; green, 0; blue, 0 }  ][line width=0.75]    (10.93,-4.9) .. controls (6.95,-2.3) and (3.31,-0.67) .. (0,0) .. controls (3.31,0.67) and (6.95,2.3) .. (10.93,4.9)   ;
\draw    (317.79,82.6) -- (317.79,141.1) ;
\draw [shift={(317.79,143.1)}, rotate = 270] [color={rgb, 255:red, 0; green, 0; blue, 0 }  ][line width=0.75]    (10.93,-3.29) .. controls (6.95,-1.4) and (3.31,-0.3) .. (0,0) .. controls (3.31,0.3) and (6.95,1.4) .. (10.93,3.29)   ;
\draw    (195.4,161.04) -- (306.2,161.04) ;
\draw [shift={(308.2,161.04)}, rotate = 180] [color={rgb, 255:red, 0; green, 0; blue, 0 }  ][line width=0.75]    (10.93,-3.29) .. controls (6.95,-1.4) and (3.31,-0.3) .. (0,0) .. controls (3.31,0.3) and (6.95,1.4) .. (10.93,3.29)   ;

\draw (170.09,61.06) node [anchor=north west][inner sep=0.75pt]    {$\Lambda _{k}^{n}$};
\draw (310.41,65.18) node [anchor=north west][inner sep=0.75pt]    {$X$};
\draw (173.2,149.26) node [anchor=north west][inner sep=0.75pt]    {$\Delta ^{n}$};
\draw (245.14,121.35) node [anchor=north west][inner sep=0.75pt]    {$\exists f_{0}$};
\draw (310.75,152.15) node [anchor=north west][inner sep=0.75pt]    {$Y$};
\draw (324.75,99.15) node [anchor=north west][inner sep=0.75pt]    {$f$};
\draw (172.75,99.15) node [anchor=north west][inner sep=0.75pt]    {$\iota $};

\end{tikzpicture}
\end{center}
of maps of simplicial sets, there exists a map $f_0$ such that the above diagram commutes.
\end{enumerate}
\end{prop}
The reader is referred to \cite[Chap. 3.3.2]{32} for a detailed treatment of this proposition. We denote the associated presentable $\infty$-category of $\Ens$ by $\ens$.
 
\begin{defi}
A differential graded Lie algebra (or briefly dgLa) over $k$ is a cochain complex $(\mathfrak{g},d)$ of $k$-vector spaces equipped with a Lie bracket $[-,-]: \mathfrak{g}^p \otimes_k \mathfrak{g}^q \rightarrow \mathfrak{g}^{p+q}$ satisfying the following conditions:
\begin{enumerate}
\item[(1)] For $x \in \mathfrak{g}^p$ and $y\in \mathfrak{g}^q$, we have $[x,y]+(-1)^{pq}[y,x]=0$.
\item[(2)] For $x \in \mathfrak{g}^p$, $y\in \mathfrak{g}^q$  and $z\in \mathfrak{g}^r$, we have
$$(-1)^{pr}[x,[y,z]]+(-1)^{pq}[y,[z,x]]+(-1)^{qr}[z,[x,y]]=0.$$
\item[(3)] The differential $d$ is of degree $1$ and is a derivation with respect to the Lie bracket. That is, for $x \in \mathfrak{g}^p$ and $y\in \mathfrak{g}^q$, 
$$d[x,y]=[dx,y]+(-1)^p[x,dy].$$
\end{enumerate}
Given a pair of dgLas $(\mathfrak{g},d)$ and $(\mathfrak{g}',d')$, a map of dgLas from $(\mathfrak{g},d)$ to $(\mathfrak{g}',d')$ is a map of chain complexes $F: (\mathfrak{g},d) \rightarrow(\mathfrak{g}',d')$ such that $$F([x,y])=[F(x),F(y)]$$ for $x \in \mathfrak{g}^p$ and $y\in \mathfrak{g}^q$.

The collection of all dgLas over $k$ forms a category, which we shall denote by $\Li_k$. 
\end{defi}
\begin{prop} The category $\Li_k$ of dgLas over $k$ admits a combinatorial model category structure where 
\begin{enumerate}
\item[($W$)] A map of dgLas $f: \mathfrak{g} \rightarrow\mathfrak{g}'$ is a weak equivalence if and only if it is a quasi-isomorphism of cochain complexes.
\item[($F$)] A map of dgLas $f: \mathfrak{g}\rightarrow \mathfrak{g}'$ is a fibration if and only if it is degree-wise surjective.
\end{enumerate}
\end{prop}
\begin{proof}
See \cite[Prop. 2.1.10]{4}.
\end{proof}
By the construction mentioned previously, we obtain the associated $\infty$-category $\mathrm{N}(\Li_k)[W^{-1}]$, denoted simply by $\Lie_k$. As an immediate consequence, we have the following.
 \begin{coro} \label{c3.2.1} The $\infty$-category $\Lie_k$ is presentable.
 \end{coro}

\begin{defi} A commutative differential graded algebra (or briefly cdga) over $k$ is a cochain complex $(A,d)$ equipped with a morphism of chain complexes (multiplication map) $\mu: A\otimes_k A \rightarrow A$ and with a $0$-cocycle $1$ (neutral element) such that
\begin{enumerate}
\item[(1)] $\mu\left ( a,\mu(b,c) \right )=\mu\left ( \mu(a,b),c \right )$ (associativity),
\item[(2)] $\mu(a,b)=(-1)^{pq}\mu(b,a)$ (commutativity),
\item[(3)] $\mu(a,1)=\mu(1,a)=a$,
\end{enumerate}
for any $a \in A^p$ and $b\in A^q$.
A morphism of cdgas is a morphism of chain complexes commuting with multiplication maps. The collection of all cdgas over $k$ forms a category, which we shall denote by $\calg_k$. 
\end{defi}

\begin{prop} The category $\calg_k$ of dgLas over $k$ possesses a combinatorial model category structure where 
\begin{enumerate}
\item[($W$)] A map of cdgas $f: A \rightarrow A'$ is a weak equivalence if and only if it is a quasi-isomorphism of cochain complexes.
\item[($F$)] A map of cdgas $f: A \rightarrow A'$ is a fibration if and only if it is degree-wise surjective.
\end{enumerate}
\end{prop}
The same construction as in the case of dgLas gives us the associated $\infty$-category $\Calg_k$ of $\calg_k$. Let us denote by $\cdga_k$ the full sub-category of $\calg_k$ consisting of cdgas $A$ with an additional augmented map $A\rightarrow k$. This sub-category inherits a combinatorial model category structure from $\calg_k$, which permits us to talk about its corresponding $\infty$-category, denoted by $\cdg_k$. Finally, we introduce a sub-category of $\cdg_k$, on which formal moduli problems are defined.
\begin{defi} A commutative differential graded augmented $k$-algebra $A \in \cdg_k$ is said to be artinian if the three following conditions hold:
\begin{enumerate}
\item[(1)] The cohomology groups $H^{n}(A)=0$ for $n$ positive and for $n$ sufficiently negative.
\item[(2)] All cohomology groups $H^{n}(A)$ are of finite dimension over $k$.
\item[(3)] $H^0(A)$ is a local artinian ring with maximal ideal $\mathfrak{m}$ and the morphism $$H^0(A)/\mathfrak{m}\rightarrow k$$ is an isomorphism.
\end{enumerate}
We denote the full sub-category of $\cdg_k$ consisting of artinian commutative differential graded augmented $k$-algebras by $\dg_k$. A morphism $A\rightarrow A'$  in $\dg_k$ is said to be small if the induced morphism $H^0(A)\rightarrow H^0(A')$ is surjective.
\end{defi}
\subsection{Chevalley-Eilenberg complex of dgLas and Koszul duality}
\begin{defi} Let $(\mathfrak{g},d)$ be a differential graded Lie algebra over a field $k$. The cone of $\mathfrak{g}$, denoted by $\Cn(\mathfrak{g})$, is defined as follows:
\begin{itemize}
\item[(1)] For each $n\in \mathbb{Z}$, the vector space $\Cn(\mathfrak{g})$ is $\mathfrak{g}^n\oplus \mathfrak{g}^{n-1}.$ A general element of $\Cn(\mathfrak{g})^n$ is of the form $$x+\epsilon y,$$ where $x\in \mathfrak{g}^n$ and $y\in \mathfrak{g}^{n-1}$ and $\epsilon$ is a formal symbol.
\item[(2)] The differential on $\Cn(\mathfrak{g})$ is given by the formula $$d(x+\epsilon y)=dx+y-\epsilon dy.$$
\item[(3)] The Lie bracket on $\Cn(\mathfrak{g})$ is given by  $$[x+\epsilon y, x'+\epsilon y']=[x,y]+\epsilon ([y,x']+(-1)^p[x,y']).$$
\end{itemize}
\end{defi}
By definition, $\Cn(\mathfrak{g})$ is also a differential graded Lie algebra. Moreover, its underlying chain complex can be identified with the mapping cone of the identity: $\mathfrak{g} \rightarrow \mathfrak{g}$. In particular, $0\rightarrow \Cn(\mathfrak{g})$ is a quasi-isomorphism of dgLas. Note that the zero map $\mathfrak{g} \rightarrow 0$ induces a map of differential graded algebras $U(\mathfrak{g})\rightarrow U(0)=k$, where $U(\mathfrak{g})$ and $U(0)$ are the universal enveloping differential graded algebras of $\mathfrak{g}$ and that of $0$, respectively. Another evident map of dgLas is the inclusion $\mathfrak{g}\rightarrow \Cn(\mathfrak{g})$. 
\begin{defi} \cite[Con. 2.2.13]{4}, \cite[Def. 1.5]{14} The \textbf{Chevalley-Eilenberg complex} of $\mathfrak{g}$ is defined to be the linear dual of the tensor product
$$U(\Cn(\mathfrak{g}))\otimes_{U(\mathfrak{g})}^{\mathbb{L}}k,$$
which we shall denote by $C(\mathfrak{g})$. 
\end{defi}
There is a natural multiplication on $C(\mathfrak{g})$. More precisely, for $\lambda \in C^p(\mathfrak{g})$ and $\mu\in C^q(\mathfrak{g})$, we define $\lambda\mu \in C^{p+q}(\mathfrak{g})$ by the formula
$$(\lambda\mu)(x_1\cdots x_n)=\sum_{S,S'}\epsilon(S,S')\lambda(x_{i_1}\cdots x_{i_m})\mu(x_{j_1}\cdots x_{i_{n-m}}),$$ where $x_i\in \mathfrak{g}_{r_i}$, the sum is taken over all disjoint sets $S=\{ i_1< \cdots < i_m\}$ and $S'=\{j_1<\cdots <j_{n-m} \}$ and $r_{i_1}+\cdots +r_{i_m}=p$, and $\epsilon(S,S')=\prod_{i\in S',j\in S,i<j}(-1)^{r_ir_j}$. This multiplication imposes a structure of cdga on $C({\mathfrak{g}})$.
\begin{prop} \label{p1.1}
With above notations, we have the followings:
\begin{itemize}
\item[(1)] The construction $\mathfrak{g} \mapsto C(\mathfrak{g})$ sends quasi-isomorphisms of dgLas to quasi-isomorphisms of cdgas. In particular, we obtain a functor between $\infty$-categories $\Lie_k \rightarrow \cdg_k^{op}$, which, by abuse of notation, we still denote by $C$.
\item[(3)] The $\infty$-functor $C$ preserves small co-limits. Thus, $C$ admits a right adjoint $D$: $\cdg_k^{op} \rightarrow \Lie_k$ to which we refer as Koszul duality. 
\item[(4)] The unit map
 $$A\overset{\simeq}{\rightarrow} CD(A)$$ is an equivalence in $\dg_k$ and $$DCD(A)\overset{\simeq}{\rightarrow} D(A)$$ in $\Lie_k$.
\end{itemize}
\end{prop}
\begin{proof} For the first three statements, see \cite[Chap. 2, Prop. 2.2.6, Prop. 2.2.7, Prop. 2.2.17]{4}. For the last one, see \cite[Chap. 4, Prop. 4.3.5]{7}.
\end{proof}
\begin{defi}\label{d3.2.3}
We say that an object $\mathfrak{g}$ in $\Lie_k$ is good if it is cofibrant with respect to the model structure on $\Lie_k$ and there exists a graded vector subspace $V \subset \mathfrak{g}$ such that 
\begin{itemize}
\item[(1)] For every integer $n$, $V^n$ is of finite dimension.
\item[(2)] For every non-positive integer $n$, $V^n$ is trivial.
\item[(3)] As a graded Lie algebra, $\mathfrak{g}$ is freely generated by $V$, i.e. $\mathfrak{g}=\free(V)$. 
\end{itemize}
Denote the full subcategory of $\Lie_k$ spanned by those good objects by $\mathcal{C}^{\circ}$.
\end{defi}
\begin{rem}\label{r3.2.3} In general the pair \[
\begin{tikzcd}[row sep=large, column sep=large]
\tikzcdset{row sep/normal=8em}
D:\cdg_k \arrow[shift right, swap]{r}{}&\Lie_k^{op}:C \arrow[shift right, swap]{l}{}
\end{tikzcd}
\] does not induce an equivalence of categories. However, its restriction to the sub-categories $\dg_k$ and $\mathcal{C}^\circ$ really does, i.e. the following pair \[
\begin{tikzcd}[row sep=large, column sep=large]
\tikzcdset{row sep/normal=8em}
D:\dg_k \arrow[shift right, swap]{r}{}&\mathcal{C}^{\circ}:C\arrow[shift right, swap]{l}{}
\end{tikzcd}
\] is indeed an equivalence for the sake \cite[Prop. 2.3.4]{4}. In addition, $\mathcal{C}^{\circ}$ contains essentially compact objects of $\Lie_k$ (cf. Definition \ref{d3.2.4} for the notion of compact object).
\end{rem}

\subsection{Mapping spaces in $\Lie_k$ and in $\cdg_k$} 
For each $n \in \mathbb{N}$, the algebraic simplex $\Delta^n$ of dimension $n$ is the sub-variety of the affine space $\mathbb{A}^{n+1}$, defined by the equation $\sum_i x_i=1$. Let $L$ and $L'$ be two dgLas then the simplicial set of morphisms from $L$ to $L'$ is the simplicial set $$\underline{\Hom}^\bigtriangleup(L,L'):[n]\mapsto \Hom_{\Lie_k}(L,L'\otimes_k C^*(\Delta^n))$$ where $C^*(\Delta^n)$ is the de Rham differential graded algebra on the algebraic simplex $\Delta^n$ and $\Hom_{\Lie_k}(L,L'\otimes_k C^*(\Delta^n))$ is the usual set of morphisms between two dgLas $L$ and $L'\otimes_k C^*(\Delta^n)$.
\begin{defi} With the above notations, the mapping space $\Map_{\Lie_k}(L,L')$  between two dgLas $L$ and $L'$ is the simplicial set $\underline{\Hom}^\bigtriangleup(QL,L')$ where $QL$ is a cofibration replacement of $L$.
\end{defi}
\begin{rem} In particular, $\pi_0(\Map_{\Lie_k}(L,L'))=\Hom_{\Li_k}(QL,L')$.

\end{rem}

The mapping space  $\Map_{\cdg_k}(A,A')$  between two cdgas $A$ and $A'$ can be defined in a very similar way.

\subsection{Formal moduli problems for $\cdg_k$} \label{sec3.2.3}

In this subsection, we shall work with the deformation context $(\cdg_k, \lbrace k\oplus k[n]\rbrace_{n\in \mathbb{Z}})$. Here, the cdga $k\oplus k[n]$ is the square extension of $k$ by $k[n]$.
\begin{defi}\label{d3.2.12} A functor $X:\; \dg_k \rightarrow \ens$ is called a formal moduli problem if the following conditions are fulfilled.
\begin{itemize}
\item[$(1)$] The space $X(k)$ is contractible.
\item[$(2)$] For every pullback diagram 
\begin{center}
\begin{tikzpicture}[every node/.style={midway}]
  \matrix[column sep={10em,between origins}, row sep={3em}] at (0,0) {
    \node(Y){$R$} ; & \node(X) {$R_0$}; \\
    \node(M) {$R_1$}; & \node (N) {$R_{01}$};\\
  };
  
  \draw[->] (Y) -- (M) node[anchor=east]  {}  ;
  \draw[->] (Y) -- (X) node[anchor=south]  {};
  \draw[->] (X) -- (N) node[anchor=west] {};
  \draw[->] (M) -- (N) node[anchor=north] {};.
\end{tikzpicture}
\end{center}
in $\dg_k$, if $\pi_0(R_0)\rightarrow \pi_0(R_{01})\leftarrow\pi_0(R_1)$ are surjective, then the diagram of spaces 

\begin{center}
\begin{tikzpicture}[every node/.style={midway}]
  \matrix[column sep={10em,between origins}, row sep={3em}] at (0,0) {
    \node(Y){$X(R)$} ; & \node(X) {$X(R_0)$}; \\
    \node(M) {$X(R_1)$}; & \node (N) {$X(R_{01})$};\\
  };
  
  \draw[->] (Y) -- (M) node[anchor=east]  {}  ;
  \draw[->] (Y) -- (X) node[anchor=south]  {};
  \draw[->] (X) -- (N) node[anchor=west] {};
  \draw[->] (M) -- (N) node[anchor=north] {};.
\end{tikzpicture}
\end{center} is also a pullback diagram.
\end{itemize}
\end{defi}

Now, we are in a position to recall the following very well-known fundamental result in derived deformation theory, proved independently by Lurie in \cite{4} and Pridham in \cite{8}.

\begin{thm}\label{t1.2}
The functor \begin{align*}
\Psi:\Lie_k&\rightarrow \; \mathcal{FMP}\\
\mathfrak{g} &\mapsto \Map_{\Lie_k}(D(-),\mathfrak{g}) 
\end{align*} induces an equivalence of $\infty$-categories between $\Lie_k$ and $\mathcal{FMP}$ where $D$ is the functor appearing in the Koszul duality (cf. Proposition \ref{p1.1} and Remark \ref{r3.2.3}).
\end{thm}

\section{Semi-prorepresentability of formal moduli problems}\label{sec3.3}

\subsection{Smooth and étale morphisms of formal moduli problems}
\begin{defi}\label{d3.3.1} Let $X$ and $Y$ be fmps and $u:\; X\rightarrow Y$ be a map between them. 
\begin{enumerate}
\item[(i)] $u$ is said to be smooth if for every small map $\phi:\; A\rightarrow B$ in $\dg_k$, i.e. $H^0(\phi): H^0(A) \rightarrow H^0(B)$ is surjective (cf. \cite[Lem. 1.1.20]{4}), the natural map $$X(A)\rightarrow X(B)\times_{Y(B)}Y(A)$$ is surjective on connected components. 
\item[(ii)] $u$ is étale if it is smooth and furthermore $$\pi_0(X(k\oplus k)) \rightarrow \pi_0(Y(k\oplus k))$$ is an isomorphism.
\end{enumerate}
\end{defi}
\begin{rem} In the definition of étaleness, the condition that $$\pi_0(X(k\oplus k)) \rightarrow \pi_0(Y(k\oplus k))$$ is an isomorphism can be weakened to only an injection because the surjectivity of this map follows from its smoothness applying to the small morphism $k\oplus k \rightarrow k$.

\end{rem}
\begin{rem} Let $\mathfrak{g}$ and $\mathfrak{h}$ be the dgLas associated to $X$ and $Y$, respectively. Then the condition that $\pi_0(X(k\oplus k)) \cong \pi_0(Y(k\oplus k))$ is equivalent to the more explicit condition that
$$\Hom_{\Lie_k}(D(k\oplus k), \mathfrak{g})\cong \Hom_{\Lie_k}(D(k\oplus k), \mathfrak{h}),$$ on the side of dgLas.

\end{rem}

\begin{prop}\label{p2.1} Using the same notations as in Definition \ref{d3.3.1}. The following conditions are equivalent:
\begin{itemize}
\item[(i)] $u$ is smooth.
\item[(ii)] for every $n>0$, the homotopy fiber of $$X(k\oplus k[n]) \rightarrow Y(k\oplus k[n])$$ is connected.
\end{itemize}
\end{prop}
\begin{proof} See \cite[Prop. 1.5.5]{4}.
\end{proof}

The following statement gives an explicit criterion for a morphism of fmps to be étale, on the side of corresponding dgLas.
\begin{prop}\label{p2.2} Let $X$ and $Y$ be fmps whose associated dgLas are $\mathfrak{g}$ and $\mathfrak{h}$, respectively and $u:\; X\rightarrow Y$ be a map between them, inducing a map $u:\; \mathfrak{g}\rightarrow \mathfrak{h}$ of dgLas. If $H^i(\mathfrak{g})\cong H^i(\mathfrak{h}) $ for any $i>0$ then  $u$ is étale.

\end{prop}
\begin{proof}
Note that we always have that $$\begin{cases}
 H^{n-i}(\mathfrak{g})= \pi_i{X(k\oplus k[n-1])} &  \\ 
  H^{n-i}(\mathfrak{h})= \pi_i{Y(k\oplus k[n-1])} & 
\end{cases}$$ for any $i,n \geq 0$. In particular,

$$\begin{cases}
 H^{n+1}(\mathfrak{g})= \pi_0{X(k\oplus k[n])}, \;   H^{n+1}(\mathfrak{h})= \pi_0{Y(k\oplus k[n])} & \text{if } n\geq0 \\ 
 H^{n}(\mathfrak{g})= \pi_1{X(k\oplus k[n])} , \; H^{n}(\mathfrak{h})= \pi_1{Y(k\oplus k[n])} & \text{if } n> 0.
\end{cases}$$ Consider the homotopy pull-back \begin{center}
\begin{tikzpicture}[every node/.style={midway}]
  \matrix[column sep={9em,between origins}, row sep={4em}] at (0,0) {
    \node(Y){$F$} ; & \node(X) {$X(k\oplus k[n])$}; \\
    \node(M) {$*$}; & \node (N) {$Y(k\oplus k[n])$};\\
  };
  
  \draw[->] (Y) -- (M) node[anchor=east]  {}  ;
  \draw[->] (Y) -- (X) node[anchor=south]  {};
  \draw[->] (X) -- (N) node[anchor=west] {};
  \draw[->] (M) -- (N) node[anchor=north] {};.
\end{tikzpicture}
\end{center}
whose corresponding homotopy fiber sequence is 
\begin{align*}
\cdots \rightarrow\pi_1(X(k\oplus k[n]))&\rightarrow  \pi_1(Y(k\oplus k[n])) \rightarrow\pi_0(F) \\ 
 &\rightarrow \pi_0(X(k\oplus k[n]))\rightarrow  \pi_0(Y(k\oplus k[n])) \rightarrow 0.
\end{align*}
By assumption we have that $$\pi_1(X(k\oplus k[n]))\rightarrow  \pi_1(Y(k\oplus k[n]))$$ and $$\pi_0(X(k\oplus k[n]))\rightarrow  \pi_0(Y(k\oplus k[n]))$$ are all isomorphisms for $n>0$. Thus, $$\pi_0(F)=0$$ and then $F$ is connected so that $u$ is smooth by Proposition \ref{p2.1}. Besides, 
$$\pi_0(X(k\oplus k))= H^{1}(\mathfrak{g})\cong  H^{1}(\mathfrak{h})= \pi_0(Y(k\oplus k)).$$ Hence, $u$ is étale.
\end{proof}
\begin{rem} The notion of smoothness and the one of étaleness are in fact a generalization of those introduced by M. Schlessinger (cf. \cite{11})
\end{rem}

\subsection{Semi-prorepresentable formal moduli problems} One of the corollaries of Theorem \ref{t1.2} is the following criterion for a fmp to be prorepresentable (cf. \cite[Cor. 2.3.6]{4}). 
 \begin{thm}\label{t2.1}
 A fmp $F$ is prorepresentable by a pro-object in $\dg_k$ if and only if the corresponding dgLa $\mathfrak{g}$ is cohomologically concentrated in degrees $[1,+\infty)$.
 \end{thm}
 However, in reality there are many fmps which are not prorepresentable due to the fact that their associated dgLas have some components in negative degrees. The typical example is the derived deformation functor $\Def_{X_0}$ of a given algebraic scheme or a compact complex manifold $X_0$. The $0^{\text{th}}$-cohomology group of the associated dgLa of $\Def_{X_0}$ is nothing but the vector space of global vector fields on $X_0$, which is not vanishing in general. This leads us to a weaker notion of prorepresentability, which in fact generalizes that of semi-universality in the classical sense.
 
 \begin{defi}\label{d2.2}  A fmp $F$ is said to be semi-prorepresentable if there exists a pro-object in $\dg_k$ and a morphism of fmps $u: \Map_{\dg_k}(A,-) \rightarrow F$ such that $u$ is étale.
 \end{defi}
 \begin{rem}\label{r2.2} In particular, if $F$ is a semi-prorepresentable fmp in the sense of Definition \ref{d2.2} then the functor of artinian rings $E:=\pi_0(F)$ is semi-prorepresentable by $H^0(QA)$ in Schlessinger's sense:\begin{enumerate}
\item[(a)] the morphism of functors $\Hom_{\widehat{\Art}_k}(H^0(QA),-)\rightarrow E$ is smooth,
\item[(b)]  $\Hom_{\widehat{\Art}_k}(H^0(QA),k[\epsilon]/(\epsilon^2))\rightarrow E(k[\epsilon]/(\epsilon^2))$ is bijective,
\end{enumerate} where $QA$ is the cofibrant replacement of $A$ (cf. \cite{11} or \cite{12} for more details).
 \end{rem}
 \subsection{A criterion for semi-prorepresentability} In this section we try to give a sufficient condition for a given fmp to be semi-prorepresentable. 
 \begin{thm} \label{t2.2} Let $F$ be a fmp whose associated dgLa $\mathfrak{g}$. Suppose that the cohomologically concentrated interval of the associated dgLa $\mathfrak{g}$ is bounded below. Then $F$ is semi-prorepresentable.  \end{thm}
 \begin{proof} Although the proof seems well-known to experts in the domain, we still write it down for the sake of completeness. In the sequel, the reader shall quickly realize that this proof does not work in general when the group action joins the game unless some cohomological finiteness assumptions are imposed on the associated dgLa.
 
 Denote $B^1(\mathfrak{g})$ and $Z^1(\mathfrak{g})$ to be the first space of boundaries and the one of cycles, respectively. We can choose the following splittings:
 $$\mathfrak{g}_1=Z^1(\mathfrak{g})\oplus E^1, \; Z^1(\mathfrak{g})=B^1(\mathfrak{g})\oplus H^1(\mathfrak{g}). $$
 Define a new dgLa $\mathfrak{k}$ $$\begin{cases}
\mathfrak{k}^i=0 & \text{ if } i\leq 0 \\ 
 \mathfrak{k}^1= E^1\oplus H^1(\mathfrak{g}) & \text{ if } i=1 \\ 
  \mathfrak{k}^i = \mathfrak{g}^i& \text{ if } i >1,
\end{cases}$$ whose Lie bracket and differential are induced by those of $\mathfrak{g}$. The natural inclusion $u: \;\mathfrak{k} \rightarrow \mathfrak{g}$ induces isomorphisms
$$H^i(\mathfrak{k}) \rightarrow H^i(\mathfrak{g}),$$ for $i>0$ by construction. For the sake of Proposition \ref{p2.2}, the corresponding map of fmps 
$$\Map_{\Lie_k}(D(-), \mathfrak{k}) \rightarrow  \Map_{\Lie_k}(D(-), \mathfrak{g})=F(-)$$ is étale. Moreover, $\mathfrak{k}$ is cohomological concentrated in $[1,+\infty)$, by construction. Thus, the fmp $\Map_{\Lie_k}(D(-), \mathfrak{k})$ is prorepresentable by a pro object in $\dg_k$, let's say $K$, i.e.
$$\Map_{\Lie_k}(D(-), \mathfrak{k})=\Map_{\cdg_k}(K,-)$$ by Theorem \ref{t2.1}. Therefore, $F$ is semi-prorepresentable, which finishes the proof.
\end{proof}

  \begin{rem} The dgLa $\mathfrak{k}$ constructed in Theorem \ref{t2.2} is unique up to quasi-isomorphisms in $\Lie_k$.
 \end{rem}

\subsection{Semi-prorepresentability and $G$-equivariant structure} In this subsection, we intend to generalize the notion of $G$-equivariant structure on versal deformations initiated by D. S. Rim in \cite{Rim} (see also Introduction), in the world of formal moduli problems.

Let $F$ be a fmp and let  $\mathfrak{g}$ be its corresponding dgLa. Suppose that $F$ is semi-prorepresentable and that $\mathfrak{g}$ is prescribed an action of some group $G$.
\begin{defi} \label{d3.3.3} $F$ is said to have a $G$-equivariant structure if there exists a pro-object $K$ in $\dg_k$ such that the following conditions are satisfied.
\begin{itemize}
\item[$(i)$] $F$ is semi-prorepresentable by $K$,
\item[$(ii)$] Denote the associated dgLa of $K$ by $\mathfrak{k}$. Then we can equip $\mathfrak{k}$ with a compatible $G$-action such that

\begin{itemize}
 \item[$(a)$]the natural morphism of dgLas $\Phi:\; \mathfrak{k} \rightarrow \mathfrak{g}$ is $G$-equivariant with respect to the prescribed $G$-action on $\mathfrak{g}$,
 \item[$(b)$]$\mathfrak{k}$  is $G$-versal in the following sense: for any $A\in \dg_k$ and any $G$-equivariant map $\phi:\; QD(A) \rightarrow \mathfrak{g}$ with respect to the given $G$-action on $\mathfrak{g}$, there exists a $G$-equivariant map $\tau:\; QD(A)\rightarrow \mathfrak{k}$  such that the following diagram commutes
 \begin{center}
\begin{tikzpicture}[every node/.style={midway}]
  \matrix[column sep={8em,between origins}, row sep={3em}] at (0,0) {
    \node(Y){$QD(A)$} ; & \node(X) {$\mathfrak{k}$}; \\
    \node(M) { }; & \node (N) {$\mathfrak{g}$};\\
  };
  \draw[dashed,->] (Y) -- (X) node[anchor=south]  {$\tau$ };
  \draw[->] (X) -- (N) node[anchor=west] {$\Phi$ };
  \draw[->] (Y) -- (N) node[anchor=north] { $\phi$};.
\end{tikzpicture}
\end{center} 
where $QD(A)$ is a cofibrant replacement of $D(A)$,
\item[$(c)$] the construction in $(b)$ is a bijection on the tangent level. In other words, $$\Hom_{\Lie_k}^G(D(k\oplus k),\mathfrak{k})\cong \Hom_{\Lie_k}^G(D(k\oplus k),\mathfrak{g})$$ where $\Hom_{\Lie_k}^G(D(k\oplus k),\mathfrak{k})$ and $\Hom_{\Lie_k}^G(D(k\oplus k),\mathfrak{k})$ are sets of $G$-equivariant maps of dgLas into $\mathfrak{g}$ and $\mathfrak{k}$ with the prescribed $G$-actions, respectively.
 \end{itemize}

\end{itemize}

\end{defi}
\begin{rem} At first glance, Definition \ref{d3.3.3} seems long, complicated and somewhat artificial in that we suppose that $F$ is already semi-prorepresentable and that we work merely on the side of dgLas. The reason is twofold. The first is that in the classical setting, we can start talking about equivariant structures only when we already know that there exists a semi-universal element.  The second one is the inspiration from Lurie's equivalence \ref{t1.2}.
\end{rem}

\begin{rem} The $G$-equivariant structure on $F$ with respect to a fixed $G$-action on its corresponding dgLa is unique up to $G$-quasi-isomorphisms.
\end{rem}
\begin{rem}If $F$ has a $G$-equivariant structure then $K$ in the above definition will naturally carry a $G$-action. So, the map $\tau:\; QD(A)\rightarrow \mathfrak{k}$ in Definition \ref{d3.3.3}(b) correspond to a $G$-equivariant map of cdgas: $QK\rightarrow A$.
\end{rem}

\begin{rem}In Definition \ref{d3.3.3}(c), $\Hom_{\Lie_k}^G(D(k\oplus k),-) $ is the hom-set in the non-derived sense.
\end{rem}
 A criterion for a semi-prorepresentable formal moduli problem to have a $G$-equivariant structure will be given by the following.
 
\begin{thm}\label{t2.3} Let $F$ be a fmp whose associated dgLa $\mathfrak{g}$ is cohomologically concentrated in $[0,+\infty)$ and $G$ be an algebraic group defined over $k$, acting on $\mathfrak{g}$. Denote $B^1(\mathfrak{g})$ and $Z^1(\mathfrak{g})$ to be the first space of boundaries and the one of cycles. Assume further that $H^1(\mathfrak{g})$ is a finite-dimensional vector space and there exists splittings 
\begin{equation}\label{e2.11}\mathfrak{g}^1=Z^1(\mathfrak{g})\oplus E^1, \; Z^1(\mathfrak{g})=B^1(\mathfrak{g})\oplus H^1(\mathfrak{g}) \end{equation}  as $G$-modules. Then $F$ admits a $G$-equivariant structure.
 \end{thm}

 \begin{proof} 
 Let $\mathfrak{g}$ be such a dgLa. Define a new dgLa $\mathfrak{k}$ $$\begin{cases}
\mathfrak{k}^i=0 & \text{ if } i\leq 0 \\ 
 \mathfrak{k}^1= E^1\oplus H^1(\mathfrak{g}) & \text{ if } i=1 \\ 
  \mathfrak{k}^i = \mathfrak{g}^i& \text{ if } i >1,
\end{cases}$$ whose Lie bracket and differential are induced by those of $\mathfrak{g}$. It is clear that $\mathfrak{k}$ inherits an algebraic $G$-action. By the proof of Theorem \ref{t2.2}, the fmp  $F=\Map_{\Lie_k}(D(-), \mathfrak{g})$ is semi-propresentable by a pro-object $K$ whose associated dgLa is exactly $\mathfrak{k}$. Moreover, the natural map of dgLas $\Phi:\; \mathfrak{k} \rightarrow \mathfrak{g}$ is $G$-equivariant, by construction. It is left to verify the $G$-versality of $\mathfrak{k}$. However, this follows immediately from the étaleness of the map $$\Map_{\Lie_k}(D(-), \mathfrak{k}) \rightarrow  \Map_{\Lie_k}(D(-), \mathfrak{g})=F$$ and the injectivity of the natural map $\Phi:\; \mathfrak{k} \rightarrow \mathfrak{g}$.\end{proof}
\begin{rem}
The splittings (\ref{e2.11}) in fact are available in several specific situations, for example, when $\mathfrak{g}$ is the Kodaira-Spencer dgLa that controls deformations of compact complex manifolds equipped with an appropriate holomorphic action of a reductive complex Lie group (cf. Lemma \ref{l3.2} below).
\end{rem}
 The following corollary is useful in cases where $\mathfrak{g}$ can be approximated by dgLas whose first graded component is finite-dimensional.
 \begin{coro}\label{c2.3} Let $F$ be a fmp whose associated dgLa $\mathfrak{g}$ is cohomologically concentrated in $[0,+\infty)$ and $G$ be a linearly reductive algebraic group defined over $k$, acting on $\mathfrak{g}$. Assume further that $H^1(\mathfrak{g})$ is a finite-dimensional vector space for each $i\geq0$ and that the following colimit is available
 \begin{equation} \label{e3.3.1}
\mathfrak{g}=\coli_i \; \mathfrak{g}(i)
 \end{equation}  where each $\mathfrak{g}(i)$ is a dgLa such that
 \begin{enumerate}
\item[(i)] any graded component $\mathfrak{g}(i)^1$ is finite-dimensional,
\item[(ii)] $\mathfrak{g}(i)$ is cohomologically concentrated in $[0,+\infty)$,
\item[(iii)] $ \mathfrak{g}(i)$ carries an algebraic $G$-action and the colimit of these $G$-actions gives back the initial $G$-action on $\mathfrak{g}$.
\end{enumerate}
 Then $F$ admits a $G$-equivariant structure.
 \end{coro}
\begin{proof}
  As usual, we first deal with the case where each $\mathfrak{g}^1$ is finite-dimensional. As before, denote $B^1(\mathfrak{g})$ and $Z^1(\mathfrak{g})$ to be the first space of boundaries and the one of cycles, respectively. Note that $B^1(\mathfrak{g})$ and $Z^1(\mathfrak{g})$ are also $G$-invariant. Since $\mathfrak{g}^1$ is a finite-dimensional $G$-module and $G$ is reductive, we can choose the following splittings: $$\mathfrak{g}^1=Z^1(\mathfrak{g})\oplus E^1, \; Z^1(\mathfrak{g})=B^1(\mathfrak{g})\oplus H^1(\mathfrak{g}) $$ as $G$-modules. Hence, $F$ admits a $G$-equivariant structure by Theorem \ref{t2.3}.

To deal with the general case, we shall make use of the assumption $(\ref{e3.3.1})$. For each dgla $\mathfrak{g}(i)$, we repeat the above procedure to obtain $\mathfrak{k}(i)$ representing the $G$-equivariant structure on $\mathfrak{g}(i)$. Finally, the desired  $\mathfrak{k}$ is nothing but $\coli_i \; \mathfrak{k}(i)$. 
\end{proof}

 \section{Applications: Derived deformations of some geometric objects}\label{sec3.4}
 
\subsection{Deformations of algebraic schemes} \label{sec3.4.1}
 
 Let $X_0$ be an algebraic scheme defined over $k$. For each $A \in \dg_k$, denote $C_A$ the category of flat morphisms of derived schemes $X\rightarrow \Spec(A)$. A morphism between two objects $X\rightarrow \Spec(A)$ and $Y\rightarrow \Spec(A)$ in $C_A$ is a commutative square 
\begin{center}
\begin{tikzpicture}[every node/.style={midway}]
  \matrix[column sep={10em,between origins}, row sep={3em}] at (0,0) {
    \node(Y){$X$} ; & \node(X) {$Y$}; \\
    \node(M) {$\Spec(A)$}; & \node (N) {$\Spec(A)$};\\
  };
  
  \draw[->] (Y) -- (M) node[anchor=east]  {}  ;
  \draw[->] (Y) -- (X) node[anchor=south]  {};
  \draw[->] (X) -- (N) node[anchor=west] {};
  \draw[->] (M) -- (N) node[anchor=north] {};.
\end{tikzpicture}
\end{center}
in $\dSc_k$. Consider the functor
\begin{align*}
\Def : \dg_k&\rightarrow \ens \\
A &\mapsto \mathrm{N}(C_A) [\text{(quasi-isomorphisms)}^{-1}]
\end{align*} where $\mathrm{N}(C_A)$ is the nerve of the category $C_A$. Let $\phi$: $A \rightarrow A'$ be a morphism in $\dg_k$, then we have an induced morphism

\begin{align*}
\Def(\phi) : \Def (A)&\rightarrow \Def(A') \\
(X\rightarrow \Spec(A)) &\mapsto (X\times_{\Spec(A)}\Spec(A') \rightarrow \Spec(A') ) 
\end{align*} which clearly preserves the quasi-isomorphisms. The fact that $X_0 \in \Def (k)$ allows us to define a new functor

$$\Def_{X_0}: \text{ }\dg_k \rightarrow \ens $$ which sends $(A \overset{\phi_A}{\rightarrow} k)$ to the homotopy fiber at $X_0$, i.e. $\Def (A)\times_{\Def (k)}X_0$ which is equivalent to the following cartesian diagram
\begin{center}
\begin{tikzpicture}[every node/.style={midway}]
  \matrix[column sep={10em,between origins}, row sep={3em}] at (0,0) {
    \node(Y){$\Def (A)\times_{\Def (k)}X_0$} ; & \node(X) {$X_0$}; \\
    \node(M) {$\Def (A)$}; & \node (N) {$\Def (k)$};\\
  };
  
  \draw[->] (Y) -- (M) node[anchor=east]  {}  ;
  \draw[->] (Y) -- (X) node[anchor=south]  {};
  \draw[->] (X) -- (N) node[anchor=west] {$i$};
  \draw[->] (M) -- (N) node[anchor=north] {$\Def (\phi_A)$};.
\end{tikzpicture}
\end{center}
Thus, $\Def_{X_0}$ is the derived deformation functor of $X_0$ and $\Def_{X_0}\in \mathcal{FMP}$. 
\begin{rem} The formal moduli problem $\Def_{X_0}$ defined as above is the natural extension of the functor of artinian rings $F_{X_0}$ discussed in the introduction.
\end{rem}

\begin{thm} If $X_0$ is either an affine scheme with at most isolated singularities or a complete algebraic variety then $\Def_{X_0}$ is semi-prorepresentable. Consequently, the classical functor of deformations $\pi_0(\Def_{X_0}(\pi_0(-)))$ of $X_0$ 
has a semi-universal element.
\end{thm}
\begin{proof} It is very well-known that the dgLa associated to $\Def_{X_0}$ is the derived global section  $\rt(X_0,\mathbb{T}_{X_0/k})$ of $\mathbb{T}_{X_0/k}$ where $\mathbb{T}_{X_0/k}$ is the tangent complex of $X_0$ over $k$ (cf. \cite[Page 1111-30]{14}). Moreover, $\rt(X_0,\mathbb{T}_{X_0/k})$ is cohomologically concentrated in $[0,+\infty)$. Therefore, $\Def_{X_0}$ is semi-prorepresentable by Theorem \ref{t2.2}. The last statement follows immediately by Remark \ref{r2.2}.
\end{proof}

\subsection{Deformations of complex compact manifolds} Let $X_0$ be a complex complex manifold and $\mathcal{T}_{X_0}$ be its holomorphic tangent bundle.  Denote by $\mathcal{A}^{p,q}$ the sheaf of differential forms of type $(p,q)$ and by $\mathcal{A}^{p,q}(\mathcal{T}_{X_0})$ the sheaf of differential forms of type $(p,q)$ with values in $\mathcal{T}_{X_0}$. Let $\mathfrak{g}$ be the following differential graded Lie algebra
$$\Gamma(X_0, \mathcal{A}^{0,0}(\mathcal{T}_{X_0}))\overset{\bar{\partial}}{\rightarrow}\Gamma(X_0, \mathcal{A}^{0,1}(\mathcal{T}_{X_0}))\overset{\bar{\partial}}{\rightarrow}\Gamma(X_0, \mathcal{A}^{0,2}(\mathcal{T}_{X_0}))\overset{\bar{\partial}}{\rightarrow}\cdots$$ with the Lie bracket defined by 
$$[\phi d\bar{z}_I, \psi  d\bar{z}_J]=[\phi,\psi]' d\bar{z}_I \wedge\bar{z}_J$$ where $\phi,\psi \in \mathcal{A}^{0,0}(\mathcal{T}_{X_0})$ are vector fields on $X_0$, $[-,-]'$ is the usual Lie bracket of vector fields, $I,J\subset \lbrace 1,\ldots, n \rbrace$ and $z_1,\ldots, z_n$ are local holomorphic coordinates. Note that $\mathfrak{g}$ is concentrated in degrees $\geq 0$. It is well-known that deformations of $X_0$ is governed by this  $\mathfrak{g}$. Furthermore if there is a reductive Lie group acting holomorphically on $X_0$, then  $\mathfrak{g}$ receives naturally an induced linear $G$-action and any $G$-equivariant deformation of $X_0$ is controlled by  $\mathfrak{g}$ equipped with this induced $G$-action (for a quick review of (equivariant) deformations of complex compact manifolds, we refer the reader to \cite{2}).  

Now, we would like to recall the classical deformation functor $\mathbf{\mathrm{MC}}_{\mathfrak{g}}$ associated to $\mathfrak{g}$, defined via the Maurer-Cartan equation (see \cite[$\S6$]{5} for more details). We have two functors:
\begin{itemize}
\item[(1)] The Gauge functor 
\begin{align*}
G_{\mathfrak{g}}:\; \Art_\mathbb{C} & \rightarrow \mathrm{Grp} \\
A &\mapsto \mathrm{exp}(\mathfrak{g}^0\otimes \mathbf{\mathrm{m}}_A)
\end{align*} where $\mathbf{\mathrm{m}}_A$ is the unique maximal ideal of $A$ and $\mathrm{Grp}$ is the category of groupoids.

\item[(2)] The Maurer-Cartan functor $MC_{\mathfrak{g}}:\; \Art_\mathbb{C} \rightarrow \se $ defined by
\begin{align*}
MC_{\mathfrak{g}}:\; \Art_\mathbb{C} & \rightarrow \mathrm{Grp} \\
A & \mapsto \left \{ x \in \mathfrak{g}^1\otimes \mathrm{m}_A\mid \overline{\partial} x+\frac{1}{2}[x,x]=0  \right \}.
\end{align*} 
\end{itemize}
For each $A$, the gauge action of $G_{\mathfrak{g}}(A)$ on the set $MC_{\mathfrak{g}}(A)$ is functorial in $A$ and gives an action of the group functor  $G_{\mathfrak{g}}$ on $MC_{\mathfrak{g}}$. This allows us to define the quotient functor \begin{align*}
\mathbf{\mathrm{MC}}_{\mathfrak{g}}:\; \Art_\mathbb{C} & \rightarrow \se \\
A & \mapsto MC_{\mathfrak{g}}(A)/G_{\mathfrak{g}}(A),
\end{align*} 
Let $\mathfrak{Def}_{X_0} :\; \Art_\mathbb{C} \rightarrow \se $ (resp. $\mathfrak{Def}_{X_0}^G :\; \Art_\mathbb{C}^G \rightarrow \se $ ) be the functor which associates to each local artinian $k$-algebra (resp. $G$-local artinian $k$-algebra) $A$, the isomorphism (resp. $G$-equivariant isomorphism) classes of flat proper morphisms of analytic spaces $X\rightarrow \Sp(A)$ with an isomorphism (resp. $G$-equivariant isomorphism) $$X\times_{\Spec(A)} \Spec(\mathbb{C}) \cong X_0.$$ The following is fundamental (cf. \cite[Thm. V.55]{5}).
\begin{thm}\label{t3.7} There is an isomorphism 
$$\mathfrak{Def}_{X_0} \cong\mathbf{\mathrm{MC}}_{\mathfrak{g}} $$ as functors of Artin rings.
\end{thm}

On one hand, the classical deformation functor $\mathbf{\mathrm{MC}}_{\mathfrak{g}}$ can be naturally extended to a formal moduli problem in Lurie's sense (cf. $\S$\ref{sec3.2.3}) via a simplicial version of the Maurer-Cartan equation (see \cite{3} for such a construction). In other words, we have a fmp \begin{equation}\label{e3.1} \mathfrak{MC}_{\mathfrak{g}}:\; \dg_\mathbb{C} \rightarrow \ens 
\end{equation} such that $$\pi_0(\mathfrak{MC}_{\mathfrak{g}})=\mathbf{\mathrm{MC}}_{\mathfrak{g}}.$$ On the other hand, there is an equivalence \begin{equation}\label{e3.2} \mathfrak{MC}_{\mathfrak{g}} \rightarrow \Map_{\Lie_{\mathbb{C}}}(D(-),\mathfrak{g}) 
\end{equation}as fmps (cf. \cite[$\S2$]{4}). Consequently, we can think of $\Map_{\Lie_{\mathbb{C}}}(D(-),\mathfrak{g})$ as a natural extension of $\mathfrak{Def}_{X_0}$ in the derived world.

\begin{thm} The fmp $\Map_{\Lie_{\mathbb{C}}}(D(-),\mathfrak{g})$ is semi-prorepresentable. Consequently, the classical functor of deformations $\mathfrak{Def}_{X_0}$
has a formal semi-universal element.
\end{thm}

\begin{proof}
The first statement follows from the fact that $\mathfrak{g}$ is concentrated in degrees $[0,+\infty)$ and that all the cohomologies $H^i(\mathfrak{g})$ are finite-dimensional vector spaces. The last statement is the immediate consequence of the following chain of isomorphisms \begin{align*}
\mathfrak{Def}_{X_0} & \cong\mathbf{\mathrm{MC}}_{\mathfrak{g}}\\ 
 & \cong \pi_0(\mathfrak{MC}_{\mathfrak{g}})\\ 
 & \cong \pi_0(\Map_{\Lie_{\mathbb{C}}}(D(-),\mathfrak{g}))
\end{align*} and of Remark \ref{r2.2}.
\end{proof}
\begin{rem} The above theorem gives an algebraic approach to produce a formal solution to the deformation problem of complex compact manifolds. The base of the formal semi-universal element can be thought of as a formal Kuranishi space in the classical sense. However, the hardest part is always to ensure that among the formal solutions, there exists at least a convergent one.
\end{rem}
\subsection{Equivariant deformations of complex compact manifolds}
Finally, we allow the group action to rejoin the game. The rest of this section is devoted to proving the existence of a formal $G$-equivariant semi-universal element for the functor $\mathfrak{Def}_{X_0}^G$. Recall that $\mathfrak{g}$ has naturally a $G$-action induced from the one on $X_0$. In order to approximate $\mathfrak{g}$, we shall make use of a $G$-equivariant version of Hodge decomposition for complex compact manifolds. 
\begin{lem}\label{l3.2}  Let $\mathfrak{g}$ be the Dolbeault complex with values in the holomorphic tangent $\mathcal{T}_{X_0}$ then we have splittings $$\mathfrak{g}^1=Z^1(\mathfrak{g})\oplus E^1, \; Z^1(\mathfrak{g})=B^1(\mathfrak{g})\oplus H^1(\mathfrak{g}).$$ In other words, $\mathfrak{g}$ satisfies the assumptions of Theorem \ref{t2.3}.
\end{lem} 
\begin{proof} We treat the case when $G$ is a compact Lie group first. Since $G$ is compact,  we can impose a $G$-invariant Hermitian metric $\left \langle \cdot,\cdot \right \rangle$ on $\mathcal{T}_{X_0}$ by means of Weyl's trick (cf. \cite[\S4]{2}). Therefore, we have a $G$-invariant metric on $\mathfrak{g}^1= \Gamma(X_0, \mathcal{A}^{0,1}(\mathcal{T}_{X_0}))$. As usual, we find the formal adjoint $\overline{\partial}^*$ of $\overline{\partial}$. Since $G$ acts on $X_0$ by biholomorphisms then the operator $\overline{\partial}$ is $G$-equivariant. By the adjoint property together with the fact that the imposed metric is $G$-invariant, we also have that $\overline{\partial}^*$ is $G$-equivariant. Hence, so is the Laplacian 
$\square:= \overline{\partial}^*\overline{\partial}+\overline{\partial}\overline{\partial}^*$. As a matter of fact, Hodge theory provides us an orthogonal decomposition
\begin{equation} \mathfrak{g}^1=\mathcal{H}^{0,1}\bigoplus \square \mathfrak{g}^1
\end{equation} as representations of $G$ and two linear operators:
\begin{enumerate}
\item[(a)] The Green operator $\mathcal{G}:$ $\mathfrak{g}^1\rightarrow\square \mathfrak{g}^1$,
\item[(b)] The harmonic projection operator $H:$ $\mathfrak{g}^1\rightarrow\mathcal{H}^{0,1} $,
\end{enumerate}
where $\mathcal{H}^{0,1}$ is the vector space of all harmonic vector $(0,1)$-forms on $X_0$ (this space can also be canonically identified with $H^1(X_0,\mathcal{T}_{X_0}$) as $G$-modules, such that for all $v\in \mathfrak{g}^1 $, we have 
\begin{equation}
v=Hv+\square \mathcal{G}v.
\end{equation}
Therefore, we can deduce the following decomposition.
\begin{equation}
\mathfrak{g}^1=\mathcal{H}^{0,1}\oplus \bar{\partial}\mathfrak{g}^{0}\oplus \bar{\partial}^*\mathfrak{g}^{2}
\end{equation} as $G$-modules. 

Finally, the case that $G$ is a general reductive complex Lie group follows from the fact that $G$ is the complexification of one of its maximal compact subgroup.
\end{proof}
\begin{thm}\label{t3.9}There exists a $G$-equivariant structure on the semi-prorepresentable object of $\Map_{\Lie_{\mathbb{C}}}(D(-),\mathfrak{g})$ with respect to the action on $\mathfrak{g}$, induced by the fixed one on $X_0$.  Consequently, the classical functor of $G$-equivariant deformations $\mathfrak{Def}_{X_0}^G$ of $X_0$ 
has a formal $G$-equivariant semi-universal element.
\end{thm}
\begin{proof}
For the sake of Theorem \ref{t2.2}, the fmp $\Map_{\Lie_{\mathbb{C}}}(D(-),\mathfrak{g})$ is semi-prorepresentable by a pro-object $K$ in $\dg_k$. Let $\mathfrak{k}$ be the corresponding dgLa of $K$. By Lemma \ref{l3.2} and Theorem \ref{t2.3}, there exists a compatible $G$-action on $\mathfrak{k}$ which is also versal in the sense mentioned therein. Equivalently, there exists a compatible $G$-action on $K$ which is versal in the following sense. For each $A\in \dg_{\mathbb{C}}$, denote by $Q(A)$  any cofibrant replacement of $A$. Then any (non-homotopic) $G$-equivariant map of dgLa: $QD(A) \rightarrow \mathfrak{g}$ which then corresponds to a $G$-equivariant map of cdgas from $QK\rightarrow A$. Note also that  $H^0(QK)$ is a pro-object in $\Art_{\mathbb{C}}^G$.

For the last statement, we claim that $\mathfrak{Def}_{X_0}^G$ is semi-prorepresentable by $H^0(QK)$ in the sense that
\begin{enumerate}
\item[(a)] the induced morphism of functors $\Hom_{\widehat{\Art}_k^G}(H^0(QK),-)\rightarrow \mathfrak{Def}_{X_0}^G$ is surjective,
\item[(b)]  $\Hom_{\widehat{\Art}_k^G}(H^0(QK),k[\epsilon]/(\epsilon^2))\rightarrow \mathfrak{Def}_{X_0}^G(\mathbb{C}[\epsilon]/(\epsilon^2))$ is bijective
\end{enumerate}
(cf. Remark \ref{r2.2} above). Let $X \rightarrow \Spec(A)$ be an element of $\mathfrak{Def}_{X_0}^G$ where $A\in \Art_{\mathbb{C}}^G$. By \cite[Thm. 3.1 and Thm. 3.2]{2}, it corresponds to a $G$-equivariant map $\Phi_{A}:\Spec(A) \rightarrow \mathfrak{g}_1$ with respect to the action on $\mathfrak{g}$, induced by the fixed one on $X_0$ such that the following conditions are satisfied:
\begin{itemize}
\item[(i)] $\Phi_A(0)=0$,
\item[(ii)] $\Phi_{A}(a)+\frac{1}{2}[\Phi_A(a),\Phi_A(a)]=0$ for all $a\in \Spec(A)$.
\end{itemize}
This is equivalent to a $G$-equivariant map $\phi_{A}:\;D(A) \rightarrow \mathfrak{g}$ by Theorem \ref{t3.7}, isomorphisms \ref{e3.1} and \ref{e3.2}. Hence, by the previous paragraph, we have that $\phi_A$ corresponds to a $G$-equivariant map of cdgas  $\sigma_A: QK \rightarrow A$. However, $A$ is concentrated in degree $0$. Thus, $\sigma_A$ can be given as a $G$-equivariant map $H^0(QK) \rightarrow A$. Hence $(a)$ is proved. Finally, $(b)$ can be deduced from the fact that 
\begin{align*}
 \Hom_{\widehat{\Art}_k}(H^0(QK),k[\epsilon]/(\epsilon^2)) &=  \pi_0(\Map_{\Lie_{\mathbb{C}}}(D(k[\epsilon]/(\epsilon^2)),\mathfrak{g}))\\ 
 &= \Hom_{\Lie_{\mathbb{C}}}(D(k[\epsilon]/(\epsilon^2)),\mathfrak{g}).
\end{align*}
This completes the proof.
\end{proof}
\begin{rem} Once again a formal version of the existence $G$-equivariant Kuranishi space shown in \cite{2} is given by a purely algebraic method except the step in which we used a $G$-equivariant version of the famous Hodge decomposition for complex compact manifolds. This reflects a natural phenomenon when dealing with analytic deformations of geometric objects, i.e. a formal solution is always somewhat easy to produce.
 \end{rem}

\end{document}